\documentclass[11pt]{article}
\usepackage{latexsym,amsfonts,amssymb,amsmath,amsthm}
\usepackage{graphicx}
\usepackage{cite}
\usepackage[usenames,dvipsnames]{color}
\usepackage{ulem}
\usepackage[colorlinks=true]{hyperref}
\hypersetup{urlcolor=blue, citecolor=red}

\newcommand{\norm}[1]{||#1||}

\begin{document}
\setlength{\baselineskip}{16pt}

\parindent 0.5cm
\evensidemargin 0cm \oddsidemargin 0cm \topmargin 0cm \textheight
22cm \textwidth 16cm \footskip 2cm \headsep 0cm

\newtheorem{theorem}{Theorem}[section]
\newtheorem{lemma}[theorem]{Lemma}
\newtheorem{proposition}[theorem]{Proposition}
\newtheorem{definition}{Definition}[section]
\newtheorem{example}{Example}[section]
\newtheorem{corollary}[theorem]{Corollary}

\newtheorem{remark}{Remark}[section]
\newtheorem{property}[theorem]{Property}
\numberwithin{equation}{section}
\newtheorem{mainthm}{Theorem}
\newtheorem{mainlem}{Lemma}

\numberwithin{equation}{section}

\def\p{\partial}
\def\I{\textit}
\def\R{\mathbb R}
\def\C{\mathbb C}
\def\u{\underline}
\def\l{\lambda}
\def\a{\alpha}
\def\O{\Omega}
\def\e{\epsilon}
\def\ls{\lambda^*}
\def\D{\displaystyle}
\def\wyx{ \frac{w(y,t)}{w(x,t)}}
\def\imp{\Rightarrow}
\def\tE{\widetilde E}
\def\tX{\widetilde X}
\def\tH{\widetilde H}
\def\tu{\widetilde u}
\def\d{\mathcal D}
\def\aa{\mathcal A}
\def\DH{\mathcal D(\tH)}
\def\bE{\overline E}
\def\bH{\overline H}
\def\M{\mathcal M}
\renewcommand{\labelenumi}{(\arabic{enumi})}

\def\disp{\displaystyle}
\def\undertex#1{$\underline{\hbox{#1}}$}
\def\card{\mathop{\hbox{card}}}
\def\sgn{\mathop{\hbox{sgn}}}
\def\exp{\mathop{\hbox{exp}}}
\def\OFP{(\Omega,{\cal F},\PP)}
\newcommand\JM{Mierczy\'nski}
\newcommand\RR{\ensuremath{\mathbb{R}}}
\newcommand\CC{\ensuremath{\mathbb{C}}}
\newcommand\QQ{\ensuremath{\mathbb{Q}}}
\newcommand\ZZ{\ensuremath{\mathbb{Z}}}
\newcommand\NN{\ensuremath{\mathbb{N}}}
\newcommand\PP{\ensuremath{\mathbb{P}}}
\newcommand\abs[1]{\ensuremath{\lvert#1\rvert}}

\newcommand\normf[1]{\ensuremath{\lVert#1\rVert_{f}}}
\newcommand\normfRb[1]{\ensuremath{\lVert#1\rVert_{f,R_b}}}
\newcommand\normfRbone[1]{\ensuremath{\lVert#1\rVert_{f, R_{b_1}}}}
\newcommand\normfRbtwo[1]{\ensuremath{\lVert#1\rVert_{f,R_{b_2}}}}
\newcommand\normtwo[1]{\ensuremath{\lVert#1\rVert_{2}}}
\newcommand\norminfty[1]{\ensuremath{\lVert#1\rVert_{\infty}}}

\title{Spatial-Homogeneity of Stable Solutions of Almost-Periodic Parabolic Equations with Concave Nonlinearity}

\author {
\\
Yi Wang\thanks{Partially supported by NSF of China No.11771414, 11471305.} \\
Wu Wen-Tsun Key Laboratory\\
School of Mathematical Science\\
University of Science and Technology of China\\
Hefei, Anhui, 230026, People¡¯s Republic of China
\\
\\
Jianwei Xiao\\
Department of Mathematics\\
University of California, Berkeley
\\ Berkeley, CA 94720-3840, USA
\\
\\
Dun Zhou\thanks{Corresponding author, E-mail address: zhd1986@mail.ustc.edu.cn \& zhoudun@njust.edu.cn, partially supported by NSF of China No.11601498 and the Fundamental Research Funds for the Central Universities No.30918011339.}\\
Department of Mathematics
\\ Nanjing University of Science and Technology
\\ Nanjing, Jiangsu, 210094, People¡¯s Republic of China
\\
\\
}
\date{}

\maketitle

\begin{abstract}
We study the spatial-homogeneity of stable solutions of almost-periodic parabolic equations. It is shown that if the nonlinearity satisfies a concave or convex condition, then any linearly stable almost automorphic solution is spatially-homogeneous; and moreover, the frequency module of the solution is contained in that of the nonlinearity.
\end{abstract}

\section{Introduction}
We consider the semilinear parabolic equation with Neumann boundary condition
\begin{equation}\label{main-equation}
\begin{split}
& u_t=\Delta u+f(t,u,\nabla u),\quad  t>0,x\in \Omega \\
& \frac{\partial u}{\partial n}\mid_{\p\Omega}=0 ,\quad  t>0
\end{split}
\end{equation}
where $\Omega\subset \mathbb{R}^n$ is a smooth bounded domain and $f:\mathbb{R}\times\mathbb{R}\times\mathbb{R}^n\to \mathbb{R};(t,u,p)\mapsto f(t,u,p)$ together with its first and second derivatives are almost periodic in $t$ uniformly for $(u,p)$ in any compact subset of $\mathbb{R}\times\mathbb{R}^n$. Such equation is ubiquitous throughout the modeling of population dynamics and population ecology. The almost periodicity of the nonlinearity $f$ captures the growth rate influenced by external effects which are roughly but not exactly periodic, or environmental forcing which exhibits different, non-commensurate periods.

In cases where $f$ is independent of $t$ (i.e., the autonomous case) or $f$ is time-periodic with period $T>0$ (i.e., the time $T$-periodic case), it has been known that stable equilibria or $T$-periodic solutions are not supposed to possess spatial variations on a convex domain. For instance, in terms of an autonomous equation on a convex domain $\Omega$ with $f$ being independent of $\nabla u$, Casten and Holland \cite{Casten} and Matano \cite{H.Matano1979} proved that any stable equilibrium is spatially-homogeneous (i.e., without any spatial structure). In other words, any spatially-inhomogeneous equilibrium on a convex domain must be unstable. Later, Hess \cite{P.Hess1987} considered the time $T$-periodic equation and showed that all stable $T$-periodic solutions are spatially-homogeneous on a convex domain $\Omega$.

When the system \eqref{main-equation} is driven by a time almost periodic forcing, there usually exist almost automorphic solutions rather than almost periodic ones. As a matter of fact, the appearance of almost automorphic dynamics is a fundamental phenomenon in almost periodically forced parabolic equations \cite{Shen2001,ShenYi95,Shen1998,SWZ,SWZ2}. We also refer to \cite{HY09,Johnson81,Johnson82,Obaya1,Obaya2,Ortega03,Yi04} on the study of almost automorphic dynamics in different types of almost-periodic differential systems. Among many others, Shen and Yi \cite{Shen1998} showed that any stable almost automorphic solution of \eqref{main-equation} is spatially-homogeneous on a convex domain $\Omega$.

Besides the convexity of the domain, the convexity or concavity of the nonlinearity $f$ in \eqref{main-equation} (i.e., the function $f(t,\cdot,\cdot):\mathbb{R}^{N+1}\to\mathbb{R}$ is convex or concave for all $t\in \mathbb{R}$) can be thought as an alternative condition which guarantees that any spatially-inhomogeneous equilibrium and time $T$-periodic solution are unstable in the autonomous case (Casten and Holland \cite{Casten}) and the time $T$-periodic case (Hess \cite{P.Hess1987}), respectively.

The present paper is mainly focusing on the almost periodically forced equation \eqref{main-equation}. We will show that, if $f(t,\cdot,\cdot):\mathbb{R}^{N+1}\to\mathbb{R}$ is convex or concave for all $t\in \mathbb{R}$, then any linearly stable almost automorphic solution $u(t,x)$ (see Definition \ref{linear-stable-solu}) of \eqref{main-equation} is spatially-homogeneous; and moreover, the frequency module of $u(t,x)$ is contained in that of $f$ (see Theorem \ref{main-result}).

Our result can be viewed as an effective supplement of the above-mentioned result in \cite{Shen1998}; for the concavity or convexity of $f$, instead of for convex domains.  It also generalizes to multi-frequency driven systems from that in the autonomous cases \cite{Casten} and time-periodic cases \cite{P.Hess1987}.

The paper is organized as follows. In Section 2, we review the basic notations and concepts involving skew-product semiflows, linearly stable and almost periodic (automorphic) functions which will be useful in our discussions. In Section 3, we prove the spatial-homogeneity of linearly stable almost automorphic solutions to \eqref{main-equation} under the assumption that the nonlinearity $f$ is concave or convex. 
\vskip 3mm

\section{Notations and Preliminary Results}
\subsection{Skew-product Semiflows and Linearly Stable Solutions}\label{semiflow1}

Let $Y$ be a compact metric space with metric $d_{Y}$ and $\mathbb{R}$ be the additive group of reals. A real {\it flow} $(Y,\mathbb{R})$ (or $(Y,\sigma)$) is a continuous mapping $\sigma:Y\times \mathbb{R}\to Y, (y,t)\mapsto y\cdot t$ satisfying: (i) $\sigma(y,0)=y$; (ii) $\sigma(\sigma(y,s),t)=\sigma(y,s+t)$ for all $y\in Y$ and $s,t\in\mathbb{R}$. A subset
$E\subset Y$ is {\it invariant} if $\sigma(y,t)\in E$ for each $y\in E$ and $t\in
\RR$, and is called {\it minimal} or {\it recurrent} if it is compact and the only non-empty compact invariant subset
of it is itself. By Zorn's Lemma, every compact and $\sigma$-invariant set contains a minimal subset. Moreover, a subset $E$ is minimal if and only if every trajectory is dense in $E$.

Let $X,Y$ be metric spaces and $(Y,\sigma)$ be a compact flow (called the base flow). Let also
$\RR^+=\{t\in \RR:t\ge 0\}$. A
skew-product semiflow $\Pi^t:X\times Y\rightarrow X\times Y$ is a semiflow of the following form
\begin{equation}\label{skew-product-semiflow}
\Pi^{t}(u,y)=(\varphi(t,u,y),y\cdot t),\quad t\geq 0,\, (u,y)\in X\times Y,
\end{equation}
satisfying (i) $\Pi^{0}={\rm Id}_X$ and (ii) the co-cycle property:
$\varphi(t+s,u,y)=\varphi(s,\varphi(t,u,y),y\cdot t)$ for each $(u,y)\in X\times Y$ and $s,t\in \RR^+$.
A subset $E\subset
X\times Y$ is {\it positively invariant} if $\Pi^t(E)\subset E$ for
all $t\in \RR^+$. The {\it forward orbit} of any $(u,y)\in X\times
Y$ is defined by $\mathcal{O}^+(u,y)=\{\Pi^t(u,y):t\ge 0\}$, and the
{\it $\omega$-limit set} of $(u,y)$ is defined by
$\omega(u,y)=\{(\widehat{u},\widehat{y})\in X\times Y:\Pi^{t_n}(u,y)\to
(\widehat{u},\widehat{y}) (n\to \infty) \textnormal{ for some sequence
}t_n\to \infty\}$.

A {\it flow extension} of a skew-product semiflow $\Pi^t$
is a continuous skew-product flow $\widehat{\Pi}^t$ such that $\widehat{\Pi}^t(u,y)=\Pi^t(u,y)$ for each $(u,y)\in X\times Y$ and $t\in \RR^+$. A compact positively invariant subset is said to admit {\it a flow extension} if the semiflow restricted to it does. Actually, a compact positively invariant set $K\subset X\times Y$ admits a flow extension if every point in $K$ admits a unique backward orbit which remains inside the
set $K$ (see \cite[part II]{Shen1998}). A compact positively invariant set $K\subset X\times Y$ for $\Pi^t$ is called {\it minimal} if it does not contain any other nonempty compact positively invariant set than itself.

Let $X$ be a Banach space and the cocycle $\varphi$ in \eqref{skew-product-semiflow} be $C^1$ for $u\in X$, that is, $\varphi$ is $C^{1}$ in $u$, and the derivative $\varphi_{u}$ is continuous in $u\in X$, $y\in Y,\,t>0$ ; and moreover, for any $v\in X$,
$$\varphi_{u}(t,u,y)v\rightarrow v \quad \textnormal{ as }\quad t\rightarrow 0^+,$$
uniformly for $(u,y)$ in compact subsets of $X\times Y$. Let $K\subset X\times Y$ be a
compact, positively invariant set which admits a flow extension. Define
$\Phi(t,u,y)=\varphi_{u}(t,u,y)$
for $(u,y)\in K,\,t\geq 0$. Then the operator $\Phi$ generates a linear skew-product semiflow $\Psi$ on $(X\times K,\mathbb{R}^{+})$ associated with \eqref{skew-product-semiflow} over $K$ as follows:
\begin{equation}\label{linearized-skew-product}
\Psi(t,v,(u,y))=(\Phi(t,u,y)v,\Pi^t(u,y)),\,\,t\geq 0,\,(u,y)\in K,\, v\in X.
\end{equation}
For each $(u,y)\in K$, define the Lyapunov exponent
$\lambda(u,y)=\limsup\limits_{t \to \infty }{\frac{\ln||\Phi(t,u,y)||}{t}},$ where
$\norm{\cdot}$ is the operator norm of $\Phi(t,u,y)$.
We call the number $\lambda_{K}={\sup}_{(u,y)\in K}\lambda(u,y)$ the {\it upper
Lyapunov exponent} on $K$.
\begin{definition}
{\rm $K$ is said to be {\it linearly stable} if $\lambda_{K}\leq 0$.}
\end{definition}

To carry out our study for the non-autonomous system \eqref{main-equation}, we embed it into a skew-product semiflow. Let $f_{\tau}(t,u,p)=f(t+\tau,u,p)(\tau \in \RR)$ be the time-translation of $f$, then the function $f$ generates a family $\{f_{\tau}|\tau \in \mathbb{R}\}$ in the space of continuous functions $C(\mathbb{R}\times \mathbb{R} \times \mathbb{R}^n,\mathbb{R})$ equipped with the compact open topology. Moreover, $H(f)$ (the closure of $\{f_{\tau}|\tau\in \mathbb{R}\}$ in the compact open topology) called the hull of $f$ is a compact metric space and every $g\in H(f)$ has the same regularity as $f$. Hence, the time-translation $g\cdot t\equiv g_{t}\,(g\in H(f))$ naturally defines a compact minimal flow on $H(f)$ and equation \eqref{main-equation} induces a family of equations associated to each $g\in H(f)$,
\begin{equation}\label{induced-equation}
\begin{split}
& u_{t}=\Delta u+g(t,u,\nabla u),\,\,\quad t>0,\quad x\in \Omega, \\
& \frac{\partial u}{\partial n}=0 \quad \text{on } \mathbb{R}^+\times \partial\Omega.
\end{split}
\end{equation}
It follows from the standard theory of parabolic equations (see, e.g. \cite{Friedman} ), for each $u_0\in C^1(\overline\Omega)$ satisfying $\frac{\partial u_0}{\partial n}$ on $\partial \Omega$, \eqref{induced-equation} admits a unique classical locally solution $\varphi(t,\cdot;u_0,g)$ with $\varphi(0,\cdot;u_0,g)=u_0$.

Hereafter, we always assume that $X$ is a fractional power space (see \cite{Hen}) associated with the operator $u\to - \Delta u$, $\mathcal {D}\to L^p(\Omega)$ such that $X\hookrightarrow C^1(\overline\Omega)$ ($X$ is compact embedded in $C^1(\overline\Omega)$), where $\mathcal{D}=\{u|u\in W^{2,p}(\Omega)\ \text{and} \ \frac{\partial u}{\partial n}|_{\partial \Omega}=0\}$, $p>n$.  For any $u\in X$ and $g\in H(f)$, \eqref{induced-equation} defines (locally) a unique solution $\varphi(t,\cdot;u,g)$ in $X$ is $C^2$ in $u$ and is continuous in $g$ and $t$ within its (time) interval of existence. In the language of dynamic systems, there is a well defined (local) skew-product semiflow $\Pi^t: X\times H(f)\to X\times H(f):$
\begin{equation}\label{semiflow}
\Pi^{t}(u,g)=(\varphi(t,\cdot;u,g),g\cdot t),\quad t>0
\end{equation}
associated with \eqref{induced-equation}. By the standard a {\it priori} estimates for parabolic equations (see \cite{Friedman,Hen}), if $\varphi(t,\cdot;u,g) (u\in X)$ is bounded in $X$ in the existence interval of the solution, then it is a globally defined classical solution. For any $\delta>0$, $\{\varphi(t,\cdot;u,g)\}$ is relatively compact, hence the $\omega$-limit set $\omega(u,g)$ is a nonempty connected compact subset of $X\times H(f)$. Moreover, by \cite{hale1988asymptotic,Hen}, $\Pi^t$ restricted to $\omega(u,g)$ is a (global) semiflow which admits a flow extension.

Let $X^+=\left\{u\in X|u(x)\geq 0, x\in \overline \Omega\right\}$. Denote by ${\rm Int} X^+$ the interior of $X^+$. Clearly, ${\rm Int} X^+\neq\emptyset$, since $\{u\in X|u(x)>0 \ \text{for $x\in \Omega$, $\frac{\partial u}{\partial n}<0$ for $x\in \partial \Omega$}\}\subset {\rm Int} X^+$. Thus, $X^+$ defines a strong ordering on $X$ as follows:
\begin{equation*}
\begin{split}
u_1\leq u_2 &\Longleftrightarrow u_2-u_1\in X^+, \\
u_1<u_2     &\Longleftrightarrow u_2-u_1\in X^+,\  u_2\neq u_1,\\
u_1\ll u_2  &\Longleftrightarrow u_2-u_1\in {\rm Int} X^+.
\end{split}
\end{equation*}
Immediately, we have the following lemma from \cite[Lemma III. 5.1]{Shen1998}.
\begin{lemma}\label{strongly-monotone}
The skew-product semiflow $\Pi^t$ in \eqref{semiflow} is strongly monotone, in the sense that: for any $(u,g)\in X\times H(f)$, $v\in X$ with $v>0$, one has $\Phi(t,u,g)v\gg 0$ for $t>0$.
\end{lemma}

\begin{definition}\label{linear-stable-solu}
{\rm  A bounded solution $u(t,x)=\varphi(t,x;u_0,g)$ of \eqref{induced-equation}($u_0\in X$) is {\it linearly stable} if it satisfies the following conditions:

(i) $\omega(u_0,g)$ is linearly stable.

(ii) Let $\Phi(t,s)$ ($t\geq s\geq 0$) be the solution operator of the following linearized equation along $u(t,x)$:
\begin{equation}
\begin{split}
& v_t=\Delta v+g_u(t,u,\nabla u)v+g_p(t,u,\nabla u)\nabla v\quad \text{in } \mathbb{R}^+\times \Omega, \\
& \frac{\partial v}{\partial n}=0 \quad \text{on } \mathbb{R}^+\times \partial\Omega.
\end{split}
\end{equation}
Then $\sup_{t\geq 0}\|\Phi(t,0)v_0\|<\infty$ for all $v_0\in X$.
}
\end{definition}

\subsection{Almost Periodic and Almost Automorphic Functions}

In this subsection, we always assume $D$ is a non-empty subset of $\RR^m$.
\begin{definition}\label{admissible}
A continuous function $f:\RR\times D \to\RR$ is said to be {\it
admissible} if for any compact subset $K\subset D$, $f$ is bounded and uniformly continuous on
$\RR\times K$. $f$ is $C^r$ ($r\ge 1$) {\it admissible} if $f$ is $C^r$ in $w\in D$ and Lipschitz in $t$, and $f$ as well as its partial derivatives to order $r$ are admissible.
\end{definition}

Let $f\in C(\RR\times D,\RR) (D\subset \RR^m)$ be admissible. Then
$H(f)={\rm cl}\{f\cdot\tau:\tau\in \RR\}$ (called the {\it hull of
$f$}) is compact and metrizable under the compact open topology (see \cite{Sell,Shen1998}), where $f\cdot\tau(t,\cdot)=f(t+\tau,\cdot)$. Moreover, the time translation $g\cdot t$ of $g\in H(f)$ induces a natural
flow on $H(f)$ (cf. \cite{Sell}).
\begin{definition}{\rm
A function $f\in C(\RR,\RR)$ is {\it almost automorphic} if for every $\{t'_k\}\subset\mathbb{R}$ there is a subsequence $\{t_k\}$
and a function $g:\mathbb{R}\to \mathbb{R}$ such that $f(t+t_k)\to g(t)$ and $g(t-t_k)\to f(t)$ pointwise. $f$ is {\it almost periodic} if for any sequence $\{t'_n\}$ there is a subsequence $\{t_n\}$ such that $\{f(t+t_n)\}$ converges uniformly. A function $f\in C(\RR\times D,\RR)(D\subset \RR^m)$ is {\it uniformly almost periodic (automorphic)} {\it in $t$}, if $f$ is both admissible and almost periodic (automorphic) in $t\in \RR$.}

\end{definition}

\begin{remark}\label{a-p-to-minial}
{\rm  If $f$ is a uniformly almost automorphic function in $t$, then $H(f)$ is always {\it minimal}, and there is a residual set $Y'\subset H(f)$, such that all $g\in Y'$ is a uniformly almost automorphic function in $t$. If $f$ is a uniformly almost periodic function in $t$, then $H(f)$ is always {\it minimal}, and every $g\in H(f)$ is uniformly almost periodic function (see, e.g. \cite{Shen1998}).}
\end{remark}
Let $f\in C(\mathbb{R}\times D,\mathbb{R})$ be uniformly almost periodic (almost automorphic) and
\begin{equation}\label{Fourier}
f(t,w)\sim \sum_{\lambda\in \mathbb{R}}a_{\lambda}(w)e^{i\lambda t}
\end{equation}
be a Fourier series of $f$ (see \cite{Shen1998,Veech1965} for the definition and the existence of a Fourier series). Then $\mathcal{S}=\{\a_{\lambda}(w)\not\equiv 0\}$ is called the Fourier spectrum of $f$ associated with Fourier series \eqref{Fourier} and $\mathcal{M}$ be the smallest additive subgroup of $\mathbb{R}$ containing $\mathcal{S}(f)$ is called the frequency module of $f$. Moreover, $\mathcal{M}(f)$ is a countable subset of $\mathbb{R}$ (see, e.g.\cite{Shen1998}).

\begin{lemma}\label{module}
Assume $f\in C(\mathbb{R}\times D,\mathbb{R})$ is a uniformly almost automorphic function, then for any uniformly almost automorphic function $g\in H(f)$, $\mathcal{M}(g)=\mathcal{M}(f)$.
\end{lemma}

\begin{proof}
See \cite[Corollary I.3.7]{Shen1998}.
\end{proof}
\section{Spatial-homogeneity of Linearly Stable Solutions}
In this section, we always assume that the function $(u,p)\mapsto f(t,u,p)$ in \eqref{main-equation} is concave (resp. convex) for each $t\in \mathbb{R}$, that is, $f(t,\lambda u_1 +(1-\lambda)u_2, \lambda p_1+(1-\lambda)p_2)\ge\, (\text{resp.} \leq) \,\,\lambda f(t,u_1,p_1)+(1-\lambda)f(t,u_2,p_2)$ for any $\lambda\in [0,1]$, $t\in \mathbb{R}$ and $(u_i,p_i)\in \mathbb{R}\times\mathbb{R}^n,i=1,2$. Clearly, $g(t,u,p)$ is also concave (resp. convex) for any $g\in H(f)$. We further assume that $f(t,\cdot,\cdot)$ is $C^2$ uniformly almost periodic. Our main result is the following theorem
\begin{theorem}\label{main-result}
Assume that $f:(t,\cdot,\cdot)\mapsto f(t,\cdot,\cdot)$ is concave {\rm (}or convex{\rm)}. Let $\varphi(t,\cdot,u_0,g)\in C^{1+\frac{\mu}{2},2+\mu}(\mathbb{R}\times \overline\Omega)$ {\rm ($\mu\in(0,1]$)} be a linearly stable almost automorphic {\rm(}almost periodic{\rm)} solution of \eqref{induced-equation}, then $\varphi(t,\cdot;u_0,g)$ is spatially-homogeneous and is a solution of
\begin{equation}\label{ODE}
u'=g(t,u,0).
\end{equation}
Moreover, $\mathcal{M}(\varphi)\subset \mathcal{M}(f)$.
\end{theorem}

Hereafter, we only consider the case when $f$ is concave, because by a transformation from $u$ to $-u$, the convexity of nonlinearity $g$ can be changed into concavity.

Let $\varphi(t,\cdot;u_0,g)\in C^{1+\frac{\mu}{2},2+\mu}(\mathbb{R}\times \overline\Omega)$ be an almost automorphic solution of \eqref{induced-equation} with $u(0)=u_0$. Then, $\omega(u_0,g)$ is an almost automorphic minimal set; and hence, $\varphi(t,x;u_0,g)$ is well defined for all $t\in\mathbb{R}$. For brevity, we write $u(t,x)=\varphi(t,x;u_0,g)$ and define the following function $c:\mathbb{R}\to \mathbb{R}$ by
\[
c(t):=\max_{x\in\overline\Omega} u(t,x), \ t\in\mathbb{R}.
\]
Let $M(t)=\{x\in \overline\Omega: u(t,x)=c(t)\}$. Then, similar as the arguments in \cite[p.327]{P.Hess1987}, $c(t)$ is a Lipchitz continuous function and hence differentiable for a.e. $t\in \mathbb{R}$; define $ \mathbb{\widetilde R}=\{t\in\mathbb{R}| c(t)\ \text{is differentiable}\}$, then $\mathbb{R}\setminus\widetilde{\mathbb{R}}$ is a set of zero measure and $c'(t)$ is continuous on $\mathbb{\widetilde R}$; and moreover, $c'(t)=u_t(t,x)$ for any $t\in\mathbb{\widetilde R}$ and $x\in M(t)$. Since $u\in C^{1+\frac{\mu}{2},2+\mu}(\mathbb{R}\times \overline\Omega)$ is an almost automporhic solution of \eqref{induced-equation}, $c'(t)\in L^\infty(\mathbb R)$. Moreover, we have the following
\begin{lemma}
$c(t)$ is an almost automorphic function.
\end{lemma}
\begin{proof}
Note that $u(t,x)$ is a uniformly almost automorphic function on $\mathbb{R}\times \overline \Omega$. Then, for any sequence $t_n\to\infty$, there are $v(t,x)\in H(u)$ (the hull of $u$) and a subsequence $\{t_{n_k}\}\subset \{ t_n\}$, such that $u(t+t_{n_k},x)\to v(t,x)$ and $v(t-t_{n_k},x)\to u(t,x)$, uniformly for $(t,x)\in I\times\overline\Omega$, where $I$ is any compact set contained in $\mathbb{R}$. In other words, for any $\epsilon>0$, there exists some $N\in\mathbb{N}$ such that
\begin{equation*}
\left\{
\begin{aligned}
& u(t+t_{n_k},x)-\epsilon <v(t,x)<u(t+t_{n_k},x)+\epsilon\\
& v(t-t_{n_k},x)-\epsilon <u(t,x)<v(t-t_{n_k},x)+\epsilon,
\end{aligned}\right.
\end{equation*}
for any $k>N$ and $(t,x)\in I\times \overline\Omega$. Therefore,
\begin{equation*}
\left\{
\begin{aligned}
& \max_{x\in \overline\Omega} u(t+t_{n_k},x)-\epsilon <\max_{x\in \overline\Omega} v(t,x)<\max_{x\in \overline\Omega} u(t+t_{n_k},x)+\epsilon\\
& \max_{x\in \overline\Omega} v(t-t_{n_k},x)-\epsilon <\max_{x\in \overline\Omega} u(t,x)<\max_{x\in \overline\Omega} v(t-t_{n_k},x)+\epsilon,
\end{aligned}\right.
\end{equation*}
that is,
\begin{equation*}
|c(t+t_{n_k})-\max_{\overline\Omega}v(t,x)|<\epsilon\ \text{ and }\ |c(t)-\max_{\overline\Omega}v(t-t_{n_k},x)|<\epsilon,
\end{equation*}
for any $k>N$ and $t\in \mathbb{R}$. This implies that $c(t)$ is an almost automorphic function.
\end{proof}

Now, we are ready to prove Theorem \ref{main-result}.

\begin{proof}[Proof of Theorem \ref{main-result}]
Let $w(t,x)=c(t)-u(t,x)$. Then, it is clear that $w(t,x)$ is a uniformly almost automorphic function and $w(t,x)\geq 0$ on $\mathbb{R}\times \overline\Omega$. Since $u(t,x)$ is a solution of \eqref{induced-equation}, denote $-\Delta$ by $A$, we have
\begin{equation}\label{differen-equa}
w_t+Aw=c'(t)-u_t+\Delta u=c'(t)-g(t,u,\nabla u)
\end{equation}
for all $t\in\widetilde{\mathbb{R}}$. Since $g$ is concave,
\begin{equation}\label{l-inequality}
g(t,c,0)\leq g(t,u,\nabla u)+\frac{\partial g}{\partial u}(t,u,\nabla u)w+\sum^n_{i=1}\frac{\partial g}{\partial p_i}(t,u,\nabla u)w_{x_j}.
\end{equation}
Let $$A(t)=A-\sum^n_{i=1}\frac{\partial g}{\partial p_i}(t,u,\nabla u)\frac{\partial}{\partial x_i}-\frac{\partial g}{\partial u}(t,u,\nabla u).$$
Together with \eqref{differen-equa}-\eqref{l-inequality}, one has
\[
w_t+A(t)w\leq c'(t)-g(t,c,0):=q(t)
\]
for all $t\in \widetilde{\mathbb{R}}$.

Since $c'(t)\in L^\infty(\mathbb{R})$, one has $q\in L^\infty(\mathbb R)$. We now divide our proof into the following two cases: (i) $q(t)\leq 0$ for a.e. $t\in \mathbb{R}$; (ii) $q(t)>0$ on a set of positive measure.

Case (i). $q(t)\leq 0$ for a.e. $t\in \mathbb{R}$. Let $h\in L^{\infty}(\mathbb{R},L^p(\Omega))$ be defined from \eqref{differen-equa} by
\begin{equation}\label{linear-equation}
w_t+A(t)w=:h(t)
\end{equation}
and $\Phi(t,s)$ be the fundamental solution associated with \eqref{linear-equation}(see Definition \ref{linear-stable-solu}). If $h\in C(\mathbb{R},L^p(\Omega))$, one can use the method of variation of constant to obtain
\begin{equation}\label{presentation}
w(t)=\Phi(t,0)w(0)+\int^t_0 \Phi(t,\tau)h(\tau)d\tau
\end{equation}
in $L^p(\Omega)$ (see, e.g. \cite[Theorem 5.2.2]{Tanabe}). For general $h\in L^{\infty}(\mathbb{R},L^p(\Omega))$ in \eqref{linear-equation}, similarly as the argument in \cite[p.328-329]{P.Hess1987}, by using the strong continuity of $\Phi(t,s)$ in $s$ and \cite[p.125, (5.33)]{Tanabe}, the following equation:
\begin{equation}\label{differen-equa2}
\frac{\partial }{\partial s}(\Phi(t,s)w(s))=\Phi(t,s)w_t(s)+\Phi(t,s)A(s)w(s)=\Phi(t,s)h(s)
\end{equation}
can be established for any $t\in\widetilde{\mathbb{R}}$. Furthermore, $\Phi(t,s)w(s)$ is in fact a Lipschitz continuous function of $s$ from $\mathbb{R}$ to $L^p(\Omega)$ (hence, $\Phi(t,s)w(s)$ is an absolutely continuous function of $s$ in $L^p(\Omega)$). By using \cite[Corollary A]{Brezis} and integrating $s$ in \eqref{differen-equa2} from $0$ to $t$, one can obtain \eqref{presentation}.

Since $h(\tau)\leq 0$ for a.e $\tau \in\mathbb{R}$, by strong positivity of $\Phi$, one has $\Phi(t,\tau)h(\tau)\leq 0$ for a.e. $\tau \in[0,t]$ ($t>0$); and hence
\[
\int^t_0 \Phi(t,\tau)h(\tau)d\tau\leq 0,\quad \forall t>0.
\]
Therefore,
\begin{equation}\label{l-inequality1}
w(t)\leq \Phi(t,0)w(0).
\end{equation}

Suppose that $u(t,x)$ is not spatially-homogeneous. Then, $w(0)>0$ in $C(\overline \Omega)$ (i.e. $w(0,x)\geq 0$ for all $x\in \overline \Omega$, and $w(0,\cdot)\neq 0$).
Noticing that the skew-product semiflow $\Pi^t$ on $X\times H(f)$ is strongly monotone (see Lemma \ref{strongly-monotone}), $\omega(u_0,g)$ admits a continuous separation (see \cite[Theorem II.4.4]{Shen1998} or \cite[Sec 3.5]{Mierczynski}) as follows: There exists continuous invariant splitting $X=X_1(v,\widetilde g)\oplus X_2(v,\widetilde g)$ ($(v,\widetilde g)\in \omega(u_0,g)$) with $X_1(v,\widetilde g)={\rm span}\{\phi(v,\widetilde g)\}$, $\phi(v,\widetilde g)\in {\rm Int}X^{+}$ and $X_2(v,\widetilde g)\cap X^+=\{0\}$ such that
\begin{equation}\label{invariant-decom}
\Phi(t,v,\widetilde g)X_1(v,\widetilde g)=X_1(\Pi^t(v,\widetilde g)),\quad \Phi(t,v,\widetilde g)X_2(v,\widetilde g)\subset X_2(\Pi^t(v,\widetilde g)).
\end{equation}
Moreover, there are $K,\gamma>0$ satisfying
\begin{equation}\label{exponential-separation}
\|\Phi(t,v,\widetilde g)|_{X_2(v,\widetilde g)}\|\le Ke^{-\gamma t}\|\Phi(t,v,\widetilde g)|_{X_1(v,\widetilde g)}\|
\end{equation} for any $t\ge 0$ and $(v,\widetilde g)\in \omega(u_0,g)$.  Write $w(0)=av_1+v_2$ with $v_1\in X_1(u_0,g)$, $\|v_1\|=1$ and $v_2\in X_2(u_0,g)$. Since $u(t,x)$ is linearly stable, $\sup_{t\geq 0}\|\Phi(t,0)v_1\|$ is bounded by Definition \ref{linear-stable-solu}.

Case (ia): $\|\Phi(t,0)v_1\|$ is bounded away from zero. In this case, there exist $M\geq m>0$ such that $m\leq\inf_{t\geq 0}\|\Phi(t,0)v_1\|\leq\sup_{t\geq 0}\|\Phi(t,0)v_1\|\leq M$. Let $\Gamma=\{\overline v| \Phi(t_n,0)v_1\to \overline v \text{ in $X$ for some $t_n\to\infty$}\}$. Since $\sup_{t\geq 0}\|\Phi(t,0)v_1\|\leq M$, by the regularity of $\Phi(t,0)$, one has $\Gamma\neq \emptyset$. {\it We further claim that $\Gamma\subset \mathrm{Int} X^+$ and $\Gamma$ is a closed subset of $X$}. In fact, for any $\overline v \in \Gamma$, one can find a sequence $\tau_n\to \infty$, such that $\Phi(\tau_n,0)v_1\to \overline v$. By virtue of \eqref{invariant-decom}, $\Phi(\tau_n,0)v_1\in X_1(\Pi^{\tau_n}(u_0,g))$. Without loss of generality, one may assume that $\Pi^{\tau_n}(u_0,g)\to(\overline u,\overline g)\in\omega(u_0,g)$. This implies that $\overline v\in X_1(\overline u,\overline g)\subset \mathrm{Int} X^+\cup\{0\}$. Note also that $\|v\|\geq m>0$. Then $\overline v\in \mathrm{Int} X^+$. Next, we prove that $\Gamma$ is closed in $X$. It suffices to prove that: if the sequence $v_n\in \Gamma$ converges to some $v^*\in X$, then $v^*\in \Gamma$. Indeed, for any positive integer $k\in \mathbb{N}$, there is $n_k>0$ such that $\|v_n-v^*\|<\frac{1}{2k}$ for any $n\geq n_k$, particularly, $\|v_{n_k}-v^*\|<\frac{1}{2k}$. Noticing that $v_{n_k}\in \Gamma$, there exists $t_{n_k}\in\mathbb{R}^+$  such that $\|\Phi(t_{n_k},0)v_1-v_{n_k}\|<\frac{1}{2k}$; and hence, $\|\Phi(t_{n_k},0)v_1-v^*\|<\frac{1}{k}$. Without loss of generality, one may assume $t_{n_k}\to \infty$ as $k\to \infty$, by letting $k\to \infty$, one has $\Phi(t_{n_k},0)v_1\to v^*$ as $t_{n_k}\to\infty$, which means $v^*\in \Gamma$. Thus we have proved the claim.

Recall that $\omega(u_0,g)$ is an almost automorphic minimal set, there is a sequence $t_n\to \infty$ such that $\Pi^{t_n}(u_0,g)\to(u_0,g)$. By choosing a subsequence, still denoted by $t_n$, one has that $\Phi(t_n,0)v_1\to v^*\in X_1(u_0,g)\cap \mathrm{Int} X^+$; in other words, there is a positive constant $a^*$ such that $v^*=a^* v_1$. Moreover, $\Phi(t,0)a^* v_1 \in \Gamma$ for any fixed $t\in\mathbb{R}^+$. Therefore, $\Phi(t_n,0)a^* v_1 \in \Gamma$. Observing that $\Phi(t,0)$ is a linear operator and $\Gamma$ is a closed set, $\Phi(t_n,0)a^* v_1=a^*\Phi(t_n,0)v_1\to (a^*)^2v_1\in \Gamma$. Similarly, by repeating this argument, we have $(a^*)^n v_1\in \Gamma$ for any $n\in \mathbb{N}$. Furthermore, by virtue of the boundedness of $\Gamma$, $a^*\leq 1$. If $0<a^*<1$, then it is not hard to see $0\in \Gamma$, a contradiction to $\Gamma\subset \mathrm{Int} X^+$. Therefore, $a^*=1$. Note that $\sup_{t\geq 0}\|\Phi(t,0)v_1\|\leq M$, by \eqref{exponential-separation}, $\|\Phi(t,0)v_2\|\to 0$ as $t\to \infty$. By letting $t=t_n$ and $n\to \infty$ in \eqref{l-inequality1}, one has
$$
w(0)\leq av_1.
$$
Therefore, $v_2\leq 0$. Observing that $X_2(u_0,g)\cap X^+=\{0\}$, $v_2=0$. Hence, $w(0)=av_1$ with $a\geq 0$. If $a>0$, then $w(0)=av_1\in \mathrm{Int}X^+$, a contradiction to that $w(0)\notin \mathrm{Int}X^+$. Thus, $a=0$ and $u(t,x)$ is spatially-homogeneous.

Case (ib): $\inf_{t\geq 0}\|\Phi(t,0)v_1\|=0$. There is a sequence $\{t_n\}\subset \mathbb{R}^+$ such that $\|\Phi(t_n,0)v_1\|<\frac{1}{n}$. When the sequence $\{t_n\}$ is bounded, there exist $t^*\in\mathbb{R}^+$ and a subsequence $t_{n_k}$ such that $t_{n_k}\to t^*$ as $k\to\infty$. Due to $\Phi(t,0)v_1$ is continuous with respect to $t$, $\Phi(t^*,0)v_1=0$, which contradicts to the strong positivity of $\Phi(t,0)$. Thus, $\{t_n\}$ is unbounded. For simplicity, we assume $t_n\to\infty$ as $n\to \infty$. Again by  \eqref{l-inequality1}, we have
\begin{equation}\label{sequence-inequa}
0\leq w(t_n)\leq a\Phi(t_n,0)v_1+\Phi(t_n,0)v_2.
\end{equation}
For such $t_n$, by choosing a subsequence if necessary, one may assume that $\Pi^{t_n}(u_0,g)\to (u^*,g^*)\in \omega(u_0,g)$ and $c(t_n)\to c^*$. Let $t_n\to\infty$ in \eqref{sequence-inequa}, one has $0\leq w^*\leq 0$ where $w^*=c^*-u^*$. So, $w^*_0=0$, that is, $u^*(x)\equiv c^*$ on $\overline \Omega$ is spatially-homogeneous. By the minimality of $\omega(u_0,g)$, every point in $\omega(u_0,g)$ is spatially-homogeneous, thus, $u_0(x)=c(0)$ on $\overline \Omega$, a contradiction.

Thus, we have proved that $u(t,x)$ is spatially-homogeneous when $q(t)\leq 0$ a.e. in $\mathbb{R}$.

\vskip 2mm
Case (ii). There is a positive measure subset $E$ in $\mathbb{R}$ such that $q(t)>0$ for all $t\in E$. In the following, we will show that this case cannot occur. Actually, this can be proved by the same arguments in \cite[p.329-330]{P.Hess1987}. For the sake of completeness, we give a detailed proof below.

Suppose that there exists such subset $E\subset \mathbb{R}$. Then one can find some $t_0\in \mathbb{\widetilde R}$ such that $q(t_0)>0$. Recall that $c'(t)$ is continuous on $\mathbb{\widetilde R}$, there are nontrivial interval $[t_1,t_2]\subset \mathbb{R}$ and $\epsilon_0>0$ satisfying $q(t)\geq \epsilon_0$ for a.e. $t\in[t_1,t_2]$. By the concavity of $g(t,\cdot,\cdot)$, we have
\begin{equation}\label{concavity-inequa}
g(t,u,\nabla u)\leq g(t,c,0)-\frac{\partial g}{\partial u}(t,c,0)(c-u)-\sum_{i=1}^{n}\frac{\partial g}{\partial p_i}(t,c,0)(c-u)_{x_i}.
\end{equation}
Let
$$
\overline A(t)=A-\sum_{i=1}^{n}\frac{\partial g}{\partial p_i}(t,c,0)\frac{\partial}{\partial x_i}-\frac{\partial g}{\partial u}(t,c,0)
$$
and
$$
\overline h(t)=\frac{d}{dt}(c-u)(t)+\overline A(t)(c-u)(t).
$$
Combing with \eqref{differen-equa} and \eqref{concavity-inequa}, one can obtain $\overline h(t)\geq q(t)\geq \epsilon_0$ for a.e. $t$ in $[t_1,t_2]$. On the other hand, similarly as in \eqref{presentation}, we have
$$
(c-u)(t_2)=\overline\Phi(t_2,t_1)(c-u)(t_1)+\int^{t_2}_{t_1}\overline\Phi(t_2,s)\overline h(s)ds,
$$
where $\overline \Phi(\cdot,\cdot)$ is the fundamental solution of $u_t=\overline A(t)u$. Note that
$$
\int^{t_2}_{t_1}\overline\Phi(t_2,s)\overline h(s)ds\geq \epsilon_0\int^{t_2}_{t_1}\overline\Phi(t_2,s)\mathbf{1} ds\gg 0\quad \text{in}\ C(\overline\Omega),
$$
where $\mathbf{1}$ is the unit constant-function. Together with $\overline\Phi(t_2,t_1)(c-u)(t_1)\geq 0$, it follows that $(c-u)(t_2)\gg 0$ in $C(\overline\Omega)$, a contradiction to the definition of $c$. So, Case (ii) cannot happen.

Therefore, we have proved that $u(t,x)\equiv \varphi(t)$ is a spatially-homogeneous solution of \eqref{induced-equation}; and moreover, it is an almost automorphic solution of \eqref{ODE}. Finally, it follows from Lemma \ref{module} and \cite[Theorem III.3.4(c)]{Shen1998}  that $\mathcal{M}(\varphi)\subset \mathcal{M}(g)=\mathcal{M}(f)$. Thus, we have completed the proof.
\end{proof}

\end{document}